\documentclass[a4paper,12pt]{amsart}
\usepackage{a4wide}
\newtheorem{prop}{Proposition}[section]
\newtheorem{lemma}[prop]{Lemma}

\newtheorem{defi}[prop]{Definition}
\newtheorem{thm}[prop]{Theorem}
\newtheorem{coro}[prop]{Corollary}
\usepackage{amssymb}
\usepackage{amsmath}
\usepackage{bm}

\def\d{\mathrm{d}}
\def\rk{\mathrm{rk}}
\def\Hol{\mathrm{Hol}}
\def\e{\varepsilon}
\def\f{\varphi}
\def\SU{\mathrm{SU}}
\def\U{\mathrm{U}}
\def\G{\mathrm{G}}
\def\SO{\mathrm{SO}}
\def\O{\mathrm{O}}

\def\Sp{\mathrm{Sp}}
\def\Ric{\mathrm{Ric}}
\def\Scal{\mathrm{Scal}}
\def\Spin{\mathrm{Spin}}
\def\grad{\mathrm{grad}}
\def\T{\mathrm{T}}
\def\Ad{\mathrm{Ad}}
\def\RM{\mathbb{R}}
\def\RM{\mathbb{R}}
\def\ZM{\mathbb{Z}}

\begin{document}

\title{Conformally related Riemannian metrics with non-generic holonomy}
\author{Andrei Moroianu}

\address{Andrei Moroianu \\ Laboratoire de Math\'ematiques de Versailles, UVSQ, CNRS, Universit\'e Paris-Saclay, 78035 Versailles, France }
\email{andrei.moroianu@math.cnrs.fr}
\date{\today}

\begin{abstract}
We show that if a compact connected $n$-dimensional manifold $M$ has a conformal class containing two non-homothetic metrics $g$ and $\tilde g=e^{2\f}g$ with non-generic holonomy, then after passing to a finite covering, either $n=4$ and $(M,g,\tilde g)$ is an ambik\"ahler manifold, or $n\ge 6$ is even and $(M,g,\tilde g)$ is obtained by the Calabi Ansatz from a polarized Hodge manifold of dimension $n-2$, or both $g$ and $\tilde g$ have reducible holonomy, $M$ is locally diffeomorphic to a product $M_1\times M_2\times M_3$, the metrics $g$ and $\tilde g$ can be written as $g=g_1+g_2+e^{-2\f}g_3$ and $\tilde g=e^{2\f}(g_1+g_2)+g_3$ for some Riemannian metrics $g_i$ on $M_i$, and $\f$ is the pull-back of a non-constant function on $M_2$.
\end{abstract}
\maketitle

\section{Introduction}

A connected Riemannian manifold $(M,g)$ of dimension $n\ge 2$ has {\em non-generic} (or reduced) holonomy if the restricted holonomy group $\Hol_0(M,g)$ of its Levi-Civita connection is strictly contained in $\SO(n)$. The Berger-Simons holonomy theorem roughly says that a Riemannian manifold with non-generic holonomy is either reducible (i.e. locally isometric to a Riemannian product), or locally symmetric (i.e. has parallel curvature) or else its restricted holonomy group is conjugate to one of the groups of the so-called Berger list (see Section \ref{s2} below for details).

The property of having non-generic holonomy is not preserved by conformal changes of the metric, except for constant (or homothetic) metric changes. Indeed, homothetic metrics have the same Levi-Civita connections, thus if $g$ has reduced holonomy, $\lambda g$ has reduced holonomy too, for every positive constant $\lambda$. 

Besides this trivial example, which will be excluded in the sequel, there are only few instances where a given conformal class on a compact manifold contains two non-homothetic Riemannian metrics which have both reduced holonomy. One of them was recently considered in \cite{mmp}, where a classification of compact manifolds carrying two non-homothetic conformally related K\"ahler metrics was obtained. The only solutions to this problem are ambik\"ahler structures (cf. \cite{acg}) in dimension $4$, or obtained by the so-called Calabi Ansatz, as $S^2$-bundles over polarized Hodge manifolds in dimensions at least $6$ (cf. \cite[Thm. 1.1]{mmp}).

Another class of examples can be constructed on products of three manifolds as follows. Let $(M_i,g_i)$ be Riemannian manifolds for $i=1,2,3$ and let $\f$ be a non-constant smooth function on $M_2$, identified with a function on $M:=M_1\times M_2\times M_3$ by pull-back. Then the Riemannian metrics $g:=g_1+(g_2+e^{-2\f}g_3)$ and $\tilde g:=e^{2\f}(g_1+g_2)+g_3$ on $M$ are conformally related, non-homothetic, and have both reducible holonomy, being product metrics. The Riemannian manifolds obtained in this way will be called {\em triple warped products} in this article.

Our main result is to show that the above examples exhaust the list of possible conformal classes on compact manifolds containing two non-homothetic metrics with reduced holonomy. More precisely, we show in Theorem \ref{main} that if $(M,g)$ is a compact Riemannian manifold, $\f\in\mathcal{C}^\infty(M)$ is non-constant and $(M,g)$ and $(M,e^{2\f}g)$ have both non-generic holonomy, then either the two metrics are both (locally) isometric to Riemannian product metrics, or else a finite covering of the manifold is ambik\"ahler or is given by the Calabi Ansatz. The reducible case is studied further in Theorem \ref{reducible} where we show that every compact manifold with two conformally related non-homothetic reducible metrics is locally isometric to a triple warped product.

The proofs go roughly as follows: assume that $M$ is compact, and both $(M,g)$ and $(M,\tilde g)$ have non-generic holonomy, with $\tilde g=e^{2\f}g$ for some non-constant smooth function $\f$. If $g$ and $\tilde g$ are both irreducible, then the Berger-Simons theorem shows that either the two metrics $g$ and $\tilde g$ are both Einstein, which is impossible by Corollary \ref{c1}, or one is K\"ahler and the other one has reduced holonomy, in which case Theorem 1.3 in \cite{mmp} applies, or one of the two metrics, say $g$, is reducible, and $\tilde g$ is Einstein irreducible. In this case $(M,g)$ has, up to a finite covering, two parallel orthogonal distributions, whose volume forms define {\em conformal Killing forms} $\tilde\omega_1$ and $\tilde\omega_2$ on the irreducible manifold $(M,g)$. Moreover, a result of Cleyton \cite{cleyton}, also proved by K\"uhnel and Rademacher \cite{kr}, shows that the conformal change of the metric only depends on one of the factors. This implies that either $\tilde\omega_1$ or $\tilde\omega_2$ is actually a {\em Killing form} on $(M,\tilde g)$ and one can apply the classification of Killing forms on manifolds with special holonomy obtained in \cite{bms} for locally symmetric spaces, in \cite{ms2} for quaternionic-K\"ahler manifolds, and in \cite{uwe1} for $G_2$ and $\Spin(7)$ manifolds.

In the case where both metrics have reducible holonomy, the key ingredient is the classification of conformal Killing forms on compact Riemannian products \cite{ms}. After passing to a finite covering, one may assume that $(M,g)$ has non-trivial, parallel, oriented distributions whose volume forms induce conformal Killing forms on $(M,\tilde g)$. Using a slight extension of \cite[Thm. 2.1]{ms}, one can show that every conformal Killing form on a reducible compact Riemannian manifold is a sum of parallel forms, Killing forms which are basic with respect to the parallel distributions, and Hodge duals of them. By doing this in both directions, and using some tricky topological arguments, one obtains that one of the $\nabla^g$-parallel distributions of $(M,g)$ is orthogonal to one of the $\nabla^{\tilde g}$-parallel distributions $(M,\tilde g)$, and the conformal change factor is constant in the direction of these two distributions, which eventually gives the desired form of the metric.

\section{Holonomy issues}\label{s2}

Since the terminology about holonomy groups is slightly confusing, let us start with giving some precise definitions.

Let $(M^n,g)$ be a connected Riemannian manifold with Levi-Civita connection $\nabla^g$. For every $x\in M$ and orthonormal frame $u:\mathbb{R}^n\to \T_x M$, the {\em holonomy group} $\Hol_u(M,g)$ is the subgroup of $\O(n)$ defined by $\Hol_u(M,g):=\Ad_u^{-1}(\Hol_x(\nabla^g))$, where $\Hol_x(\nabla^g)$ is the holonomy group of $\nabla^g$ at $x$ and $\Ad_u:\O(n)\to \mathrm{Iso}(\T_xM)$ is defined by 
$$\Ad_u(A):=u\circ A\circ u^{-1},\qquad\forall\ A\in \O(n).$$

Since all holonomy groups are conjugated, we will simply write $\Hol(M,g)\subset \O(n)$ by choosing one fixed frame $u$, being understood that $\Hol(M,g)$ is only defined up to conjugation (of course all statements below are invariant under conjugation).

\begin{defi} 1. The {\em restricted} holonomy group $\Hol_0(M,g)\subset \SO(n)$ is the connected component of the identity in $\Hol(M,g)$.

2. The metric $g$ has {\em reduced holonomy} if its restricted holonomy group $\Hol_0(M,g)$ is non-generic, i.e. strictly contained in $\SO(n)$.

3. The metric $g$ has {\em reducible holonomy} if the standard representation of $\Hol_0(M,g)$ on $\mathbb{R}^n$ is reducible. 
\end{defi}

Note that if $(\hat M,\hat g)$ denotes the universal cover of $(M,g)$, then $\Hol_0(M,g)=\Hol(\hat M,\hat g)$. 

We now introduce three disjoint classes of Riemannian manifolds with reduced holonomy which will be relevant for our study.

\begin{defi} A connected Riemannian manifold $(M^n,g)$ is called of:
\begin{itemize}
\item {\em Type K}, if $n$ is even, $\Hol_0(M,g)=\U(n/2)\subset \SO(n)$ and $(M,g)$ is not locally symmetric.
\item {\em Type E}, if $(M,g)$ is locally symmetric, irreducible, with non-constant sectional curvature, or if its restricted holonomy group $\Hol_0(M,g)$ is (conjugated to) one of the following subgroups of $\SO(n)$: $\SU(n/2)$ for $n$ even, $\Sp(n/4)$ for $n$ multiple of $4$, $\Sp(n/4)\Sp(1)$ for $n\ge 8$ and multiple of $4$, $\G_2$ for $n=7$ or $\Spin(7)$ for $n=8$.
\item {\em Type P}, if the metric $g$ has reducible holonomy.
\end{itemize}
\end{defi}

The terminology is justified by the fact that Riemannian manifolds of type K are (locally) K\"ahler, those of type E are Einstein, and those of type P are (locally) Riemannian products.

An immediate consequence of the Berger-Simons holonomy theorem \cite[p. 300]{besse} is that $(M,g)$ has reduced holonomy if and only if it is of one of the three types above. 

For later use, we now prove the following rather folklorical result:

\begin{lemma}\label{hol} Let $(M^n,g)$ be a compact Riemannian manifold of dimension $n\ge 2$. 
\begin{itemize}
\item[$(i)$] If $n=2m$ and $\Hol_0(M,g)=\U(m)$ then either $(M,g)$ or a double covering of it has full holonomy equal to $\U(m)$.
\item[$(ii)$] If $n=4q$ and $\Hol_0(M,g)=\Sp(q)\Sp(1)$, then either $(M,g)$ or a double covering of it has full holonomy equal to $\Sp(q)\Sp(1)$.
\item[$(iii)$] If $\Hol_0(M,g)$ is reducible, then a finite covering of $(M,g)$ has a non-trivial parallel distribution.
\end{itemize}
\end{lemma}

\begin{proof} $(i)$ The universal cover $(\hat M,\hat g)$ of $(M,g)$ has holonomy $\U(m)$. The fundamental group $\pi_1(M)$ acts on $(\hat M,\hat g)$ by isometries. Since the standard representation of $\U(m)$ on $\Lambda^2(\RM^{2m})$ has exactly one trivial summand (spanned by the fundamental 2-form), it follows that every parallel $2$-form on $(\hat M,\hat g)$ is a multiple of the K\"ahler form $\hat \Omega\in\Omega^2 M$. Consequently, there exists a group morphism $\e:\pi_1(M)\to \ZM/2$ such that for every $\gamma\in\pi_1(M)$ we have $\gamma^*\hat \Omega=\e(\gamma)\hat\Omega$. Then $\hat M/\ker(\e)$ has full holonomy $\U(m)$ and is either equal to $M$, if $\e$ is trivial, or to a two-sheeted covering of $M$, otherwise.

$(ii)$ The proof is similar, and follows from the fact that the standard representation of $\Sp(q)\Sp(1)$ on $\Lambda^4(\RM^{4q})$ has exactly one trivial summand, spanned by the so-called Kraines form, whose stabilizer in $\SO(4q)$ is $\Sp(q)\Sp(1)$.

$(iii)$ If $(M,g)$ is flat, by Bieberbach's theorem it is finitely covered by a flat torus (which has parallel distributions of any rank). We assume for the rest of the proof that $(M,g)$ is not flat.
The universal cover $(\hat M,\hat g)$ of $(M,g)$ has reducible holonomy, so by \cite[Theorem IV.5.4]{kn1}, the tangent bundle of $\hat M$ splits in an orthogonal direct sum $\T\hat M=T_0\oplus\ldots \oplus T_k$ of parallel sub-bundles and the holonomy group of $(\hat M,\hat g)$ satisfies $\Hol(\hat M,\hat g)=H_1\times \ldots\times H_k$, where for every $i,j\ge 1$ with $j\ne i$, $H_i$ acts irreducibly on $T_i$ and trivially on $T_j$ ($T_0$ being the flat component).  Moreover this
decomposition is unique up to a permutation of the set $\{1,\ldots,k\}$ (such permutations may occur if some of the $H_i$-representations on $T_i$ are isomorphic).
Note that $T_0$ may or may not be reduced to $0$, but by our non-flatness assumption, there are at least two non-trivial summands $T_i$ in the above decomposition.

The elements of the fundamental group $\pi_1(M)$ act on $(\hat M,\hat g)$ by isometries, so there exists a group morphism $\sigma:\pi_1(M)\to \mathfrak{S}_k$ such that $\gamma_*(T_i)=T_{\sigma(\gamma)(i)}$ for every $\gamma\in\pi_1(M)$. Every $\gamma\in\ker(\sigma)$ preserves the decomposition $\T\hat M=T_0\oplus\ldots \oplus T_k$, so the parallel distributions $T_i$ define parallel distributions on $(\hat M,\hat g)/\ker(\sigma)$, which is a covering of $(M,g)$ with finite deck transformation group $\pi_1(M)/\ker(\sigma)$, isomorphic to a subgroup of $\mathfrak{S}_k$.
\end{proof}

\section{Conformal Killing vector fields} \label{sconf}

Let $(M,g)$ be a connected Riemannian manifold of dimension $n$. A vector field $\xi$ on $M$ is called {\it conformal Killing} if the trace-free part $\left(\mathcal{L} _{\xi} g\right) _0$ of the Lie derivative $\mathcal{L} _{\xi} g$ of $g$ along $\xi$ vanishes, or equivalently, if $\mathcal{L} _{\xi} g = f \, g$ for some function $f$, depending on $\xi$ and $g$. This condition is clearly independent on conformal changes of the metric. 
A vector field $\xi$ is called {\em Killing} with respect to the Riemannian metric $g$ if $\mathcal{L} _{\xi} g = 0$. 

The following result although not explicitely stated, is due to Lichnerowicz \cite[\textsection 85]{l}, Obata \cite{obata}, Nagano and Yano \cite{n,ny}, cf. also \cite[Prop. 2.2]{gm} for a short proof:

\begin{prop}    \label{p1}
Assume that $(M^n,g)$ is a compact oriented Einstein manifold carrying a conformal vector field which is not Killing. Then $(M,g)$ is, up to constant rescaling, isometric to the round sphere $\mathbb{S}^n$.
\end{prop}

\begin{coro}\label{c1}
Assume that $M^n$ $(n\ge 2)$ is a compact manifold carrying two conformally related metrics $g$ and $\tilde g=e^{2\f}g$ which are both Einstein. Then, either $\f$ is constant (i.e. the two metrics are homothetic), or $(M,g)$ is homothetic to the round sphere $\mathbb{S}^n$.
\end{coro}

\begin{proof} The formula relating the trace-free Ricci tensors $\mathrm{Ric}_0^g$ and $\mathrm{Ric}_0^{\tilde g}$ of $g$ and $\tilde g$ reads (cf. \cite[p. 59 e)]{besse}):
\begin{equation}\label{ricci}
\mathrm{Ric}_0^{\tilde g}=\mathrm{Ric}_0^g-(n-2)(\nabla^g\d \f-\d \f\otimes\d \f)-\frac{n-2}n(\Delta^g \f+|\d \f|_g^2)g.
\end{equation}
Since $g$ and $\tilde g$ are both Einstein, this gives 
$$\nabla^g\d \f-\d \f\otimes\d \f=f\, g,\qquad\hbox{where}\ f:=-\frac{1}n(\Delta^g \f+|\d \f|_g^2).$$
This relation can be equivalently written
\begin{equation}\label{ng}\nabla^g(e^{-\f}\d \f)=(e^{-\f}f)\, g.\end{equation}
Since for every $1$-form $\alpha$, the symmetric part of $\nabla^g\alpha$ is equal to $\frac12\mathcal{L} _{\alpha^\sharp} g$, \eqref{ng} shows that the metric dual $\xi$ of the 1-form $e^{-\f}\d \f$ is conformal Killing. By Proposition \ref{p1}, either $(M,g)$ is homothetic to $\mathbb{S}^n$, or $\xi$ is Killing. In the latter case, we get $e^{-\f}f=0$, thus $\Delta^g \f+|\d \f|_g^2=0$, which shows that $\d \f=0$ by integrating over $M$.
\end{proof}

For later use, we state here the following similar result:
\begin{lemma}\label{1e}
Let $(M^n,g)$ $(n\ge 2)$ be a compact Riemannian manifold carrying a non-trivial parallel vector field $\xi$. If the metric $\tilde g:=e^{2\f}g$ on $M$ is Einstein for some smooth function $\f$, then $\f$ is constant.
\end{lemma}
\begin{proof} By replacing $M$ with a double cover if necessary, we may assume that it is oriented. The vector field $\xi$ is conformal Killing with respect to $\tilde g$, so either $(M,\tilde g)$ is homothetic to the standard sphere, or $\xi$ is Killing with respect to $\tilde g$. The former case is impossible since $g(\xi,\cdot)$ is a harmonic $1$-form on $(M,g)$, so the first Betti number of $M$ is non-vanishing. Consequently $0=\mathcal{L}_\xi\tilde g=2\xi(\f)e^{2\f} g$, showing that $\xi(\f)=0$. Denoting by $\eta:=\d\f^\sharp$ the dual vector field of $\d\f$ with respect to $g$, we have
$$\nabla^g_\xi\eta=\nabla^g_\xi\eta-\nabla^g_\eta\xi=\mathcal{L}_\xi\eta=(\mathcal{L}_\xi\d\f)^\sharp=0,$$
thus showing that 
\begin{equation}\label{nabla0}\nabla^g_\xi\d\f=0\qquad\hbox{and}\qquad\d\f(\xi)=0.\end{equation}

We denote by $\Scal^{\tilde g}$ and $\Scal^g$ the scalar curvatures of $(M,\tilde g)$ and $(M,g)$ respectively. Since $\mathrm{Ric}_0^g=\Ric^g-\frac1n\Scal^g\, g$, and $\mathrm{Ric}_0^{\tilde g}=0$, we obtain from \eqref{ricci}:
\begin{equation}\label{ricci1}
0=\Ric^g-\frac1n\Scal^g\, g-(n-2)(\nabla^g\d \f-\d \f\otimes\d \f)-\frac{n-2}n(\Delta^g \f+|\d \f|_g^2)\, g.
\end{equation}
Plugging $\xi$ into this formula and using \eqref{nabla0} together with the fact that $\Ric(\xi)=0$, yields
\begin{equation}\label{ricci2}
\Scal^g=-(n-2)(\Delta^g \f+|\d \f|_g^2).
\end{equation}
On the other hand, the scalar curvatures of $g$ and $\tilde g$ are related by (cf. \cite[p. 59 f)]{besse}):
\begin{equation}\label{ricci3}
e^{2\f}\Scal^{\tilde g}=\Scal^g+2(n-1)\Delta^g \f-(n-2)(n-1)|\d \f|_g^2,
\end{equation}
whence taking \eqref{ricci2} into account:
\begin{equation}\label{ricci4}
e^{2\f}\Scal^{\tilde g}=n(\Delta^g \f-(n-2)|\d \f|_g^2).
\end{equation}
This formula shows that the constant $\Scal^{\tilde g}$ is non-negative at a point where $\f$ attains its maximum, and non-positive at a point where $\f$ attains its minimum. Thus $\Scal^{\tilde g}=0$, so 
\begin{equation}\label{delta} \Delta^g \f=(n-2)|\d \f|_g^2.\end{equation} 
Integrating this equation over $M$ with respect to the volume form $\d\mu^g$ we obtain 
$$0=\int_M\Delta^g \f\,\d\mu^g=(n-2)\int_M|\d \f|_g^2\,\d\mu^g,$$
thus proving that $\f$ is constant on $M$ if $n\ge 3$. The same conclusion holds for $n=2$, since in this case \eqref{delta} shows directly that $\f$ is harmonic, thus constant.
\end{proof}

\section{Conformal Killing forms} \label{sck} 

In this section we review some classification results about Killing forms on manifolds with reduced holonomy, which will be crucial for the proof of Theorems \ref{main} and \ref{reducible}. 

\begin{defi} \label{d41} Let $(M,g)$ be a Riemannian manifold with Levi-Civita connection $\nabla^g$. 
A {\em conformal Killing (or twistor) form} on $(M,g)$ is a $p$-form $\psi\in\Gamma(\Lambda^{p}M)$ which satisfies 
\begin{equation}\label{killing}
\nabla^g_X\psi=\tfrac{1}{  p+1}X\lrcorner \d\psi -\tfrac{1}{  n-p+1}X^\flat\wedge\delta^g \psi,
\end{equation}
for all vector fields $X$, where $X^\flat:=g(X,\cdot)$ denotes the metric dual of $X$ and $\delta^g$ denotes the co-differential defined by the metric $g$. 
If $\psi$ is in addition co-closed, it is called a {\em Killing $p$-form}. 
This is equivalent to $\nabla^g\psi\in\Gamma(\Lambda^{p+1}M)$ or to 
$X \lrcorner \nabla^g_X \psi = 0$ for any vector field $X$.
\end{defi}

Conformal Killing forms have the following well known conformal invariance property:

\begin{lemma}[cf. e.g. \cite{bc}]\label{conf-inv} Let $g$ and $\tilde g:=e^{2\f}g$ be two conformally related metrics on a manifold $M$. Then $\psi$ is a conformal Killing $p$-form on $(M,g)$ if and only if $\tilde\psi:=e^{(p+1)\f}\psi$ is a conformal Killing $p$-form on $(M,\tilde g)$.
\end{lemma}
\begin{proof} We first compute
\begin{equation}\label{dpsi}\d\tilde\psi =e^{(p+1)\f}\left( \d\psi +(p+1) \d \f\wedge \psi\right).
\end{equation}
The Levi-Civita connections $\nabla^g$ and $\nabla^{\tilde g}$ of $g$ and $\tilde g$ are related by
$$\nabla^{\tilde g}_X Y = \nabla^g_XY  + \d \f (X) Y + \d \f(Y) X - g(X, Y)\, \grad_g\f
$$
where  $\grad_g\f$ is the gradient of $\f$ with respect to  $g$ (cf. \cite[Th. 1.159]{besse}). This immediately shows that for every $p$-form $\psi$ and tangent vector $X$, the following relation holds:
$$\nabla^{\tilde g}_X\psi = \nabla^g_X\psi  -p\, \d \f (X) \psi + X^\flat\wedge (\grad_g\f\lrcorner \psi)- \d \f\wedge (X\lrcorner\psi),
$$
whence
\begin{equation}\label{nablapsi}\nabla^{\tilde g}_X\tilde\psi =e^{(p+1)\f}\left( \nabla^g_X\psi + \d \f (X) \psi + X^\flat\wedge (\grad_g\f\lrcorner \psi)- \d \f\wedge (X\lrcorner\psi)\right),
\end{equation}

If $\{e_i\}$ is a local orthonormal frame with respect to $g$, then $\{\tilde e_i:=e^{-\f}e_i\}$ is a local orthonormal frame with respect to $\tilde g$, and thus the co-differentials of $g$ and $\tilde g$ are related by
\begin{equation}\label{deltapsi}\delta^{\tilde g}\tilde\psi=-\sum_{i=1}^n\tilde e_i\lrcorner\nabla^{\tilde g}_{e_i}\tilde\psi=e^{(p-1)\f}\left( \delta^g\psi -(n-p+1) \grad_g\f\lrcorner \psi\right).\end{equation}
Using \eqref{dpsi}, \eqref{nablapsi} and \eqref{deltapsi}, together with the fact that the metric dual of $X$ with respect to $\tilde g$ is $X^{\flat_{\tilde g}}=e^{2\f}X^\flat$, we obtain
$$\nabla^{\tilde g}_X\tilde\psi-\tfrac{1}{  p+1}X\lrcorner \d\tilde\psi +\tfrac{1}{  n-p+1} X^{\flat_{\tilde g}}\wedge\delta^{\tilde g}\tilde\psi\ =e^{(p+1)\f}\left( \nabla^g_X\psi-\tfrac{1}{  p+1}X\lrcorner \d\psi +\tfrac{1}{  n-p+1}X^\flat\wedge\delta^g\psi\right),$$
thus proving the claim.
\end{proof}

Assume that $M$ is oriented, and denote by $*$ the Hodge operator.
From the general identities
\begin{equation} \label{identity} * (X ^{\flat} \wedge \psi) = (- 1) ^p \, X \lrcorner  * \psi,\qquad *(*\psi)=(-1)^{p(n-p)}\psi,\qquad \delta^g\psi=(-1)^{n(p-1)-1}*\d *\psi \end{equation}
which hold for any vector field $X$ and any $p$-form $\psi$ on $M$, we deduce that the Hodge operator maps conformal Killing $p$--forms into conformal Killing 
$(n-p)$--forms. In particular, if $\psi$ is a closed conformal Killing form, $*\psi$ is a Killing form.

Killing forms on compact manifolds with reduced holonomy have been  recently studied in a series of papers \cite{bms}, \cite{ms1}, \cite{ms2}, \cite{ms}, \cite{uwe} and \cite{uwe1}. The following proposition summarizes some of the results of these papers.

\begin{prop}\label{prop}
A Killing form $\psi$ of degree $p\ge 2$ on a compact Riemannian manifold  $(M,g)$ is automatically $\nabla^g$-parallel provided that one of the following conditions holds:
\begin{enumerate}
\item[$(i)$] $(M,g)$ is K\"ahler;
\item[$(ii)$] $(M,g)$ is quaternion-K\"ahler;
\item[$(iii)$] $(M,g)$ has holonomy contained in $\G_2$ for $n=7$ or $\Spin(7)$ for $n=8$;
\item[$(iv)$] $(M,g)$ is locally symmetric, irreducible and has non-constant sectional curvature.
\end{enumerate}
\end{prop}
\begin{proof}
The first three assertions follow from  \cite[Lemma 4.2]{bms},  \cite[Thm. 6.1]{ms2} and  \cite[Thm. 1.1 and Thm. 1.2]{uwe1} respectively. The proof of $(iv)$ is implicitly contained in \cite[Thm. 1.1]{bms}, where one further assumes that $(M,g)$ is a simply connected symmetric space of compact type. This assumption is in fact superfluous in the irreducible case, since \cite[Lemma 4.3]{bms} actually shows that if an irreducible locally symmetric space $(M,g)$ carries a non-parallel Killing $p$-form with $p\ge 2$, then its Weyl tensor vanishes identically, thus $(M,g)$ has constant sectional curvature (being Einstein).
\end{proof}

We finally consider one further situation where conformal Killing forms can be classified. Assume that $(M,g)$ is a Riemannian manifold whose tangent bundle decomposes in an orthogonal direct sum $\T M=T_1\oplus T_2$ of $\nabla^g$-parallel distributions. A $p$-form $\psi$ on $M$ is called {\em basic} with respect to $T_1$ if $X\lrcorner\psi=0$ and $\nabla^g_X\psi$=0 for every $X\in T_2$. Of course, by the local de Rham decomposition theorem, every point has a neighbourhood isometric to a Riemannian product, and basic forms are just pull-backs of forms on the factors. We then have

\begin{prop}\label{ckp} If $(M,g)$ is a compact oriented Riemannian manifold whose tangent bundle has an orthogonal parallel splitting $\T M=T_1\oplus T_2$, then every conformal Killing form on $(M,g)$ is a sum of parallel forms, basic Killing forms with respect to $T_1$ or $T_2$, and Hodge duals of them.
\end{prop}
\begin{proof} If $(M,g)$ is a Riemannian product, this is exactly the statement of \cite[Thm. 2.1]{ms}. Although the situation needed here is slightly more general, the same proof continues to hold. One defines the partial exterior derivatives $\d_1$, $\d_2$ and co-differentials $\delta_1$, $\delta_2$ by the same formulas as in the Riemannian product case using local orthonormal bases of $T_i$, and one easily checks using Stokes' formula that $\delta_i$ is still the formal adjoint of $\d_i$ for $i=1,2$. The rest of the proof from \cite[Thm. 2.1]{ms} is unchanged.
\end{proof}

\section{Conformally related metrics with reduced holonomy}\label{s5}

We are now in position to state our first main result:

\begin{thm}\label{main} Let $M^n$ be a compact connected manifold carrying two conformally related non-homothetic Riemannian metrics $g$ and $\tilde g$ such that $(M,g)$ and $(M,\tilde g)$ have reduced holonomy. Then either $g$ and $\tilde g$ have both reducible holonomy, or up to a finite covering $(M,g,\tilde g)$ is an 
ambik\"ahler structure for $n=4$ or is obtained from the Calabi Ansatz for $n\ge 6$.
\end{thm}

\begin{proof} By assumption we have $\tilde g=e^{2\f}g$ for some non-constant function $\f$ on $M$. 

Assume first that $(M,\tilde g)$ is of type K. Lemma \ref{hol} (i) shows that after replacing $M$ with a double covering if necessary, there exists a complex structure $J$ on $M$ such that $(M,\tilde g,J)$ is K\"ahler. Then $(M,g, J)$ is a globally conformally K\"ahler manifold with non-generic holonomy. Using the classification of the possible holonomy groups of compact locally conformally K\"ahler manifolds \cite[Theorem 1.3]{mmp} we see that up to a finite covering, either $n=4$ and $(M,g,\tilde g)$ is ambik\"ahler in the sense of  \cite{acg}, or $n\ge 6$ and $(M,g,\tilde g)$ is obtained from the Calabi Ansatz, or $(M,g)$ is obtained from the construction described in \cite[Theorem 4.6]{mmp}. In the latter case, the universal covering of $(M,g)$ is isometric to $(\mathbb{R}^2\times N,\d s^2+\d t^2+ e^{2c(t)}g_N)$ for some K\"ahler manifold $(N,g_N)$ and some real function $c$, whose differential is equal to the Lee form $-\d \f$ of $(M,g,J)$. Up to a constant factor, we thus have 
$$\tilde g=e^{2\f}g=e^{-2c(t)}g=e^{-2c(t)}(\d s^2+\d t^2)+ g_N,$$
so both $g$ and $\tilde g$ have reducible holonomy. This contradicts the fact that $(M,\tilde g)$ is of type K, and thus has irreducible holonomy. By symmetry, the same argument applies if $(M,g)$ is of type K. 

If $g$ and $\tilde g$ are of type P, there is nothing to prove. 

If $g$ and $\tilde g$ are of type E, then they are both Einstein and not locally isometric to the round sphere (since $\Hol(\mathbb{S}^n)=\SO(n)$). By Corollary \ref{c1}, $g$ and $\tilde g$ are homothetic, which contradicts our assumption.

By symmetry, it remains to study one last case: $g$ is of type P and $\tilde g$ is of type E. Lemma \ref{hol} (iii) shows that, after replacing $M$ with a finite covering if necessary, we may assume that $(M,g)$ is oriented, and the tangent bundle of $M$ has non-trivial decomposition $\T M=T_1\oplus T_2$
where $T_1$ and $T_2$ are oriented, mutually orthogonal, and $\nabla^g$-parallel distributions.

\begin{lemma}\label{d1} The ranks of $T_1$ and $T_2$ are larger than $1$. 
\end{lemma}
\begin{proof} Assume for instance that $\rk(T_1)=1$. Then $T_1$ is a trivial line bundle (being oriented), and thus has a section $\xi$ of unit length. Since $T_1$ is preserved by $\nabla^g$, we have $\nabla^g\xi=0$, so Lemma \ref{1e} shows that the conformal factor $e^{2\f}$ must be constant, which contradicts the assumption that the metrics $g$ and $\tilde g$ are non-homothetic.
\end{proof}

\begin{lemma}\label{d2} The conformal factor $\f$ is constant along $T_1$ or $T_2$. 
\end{lemma}
\begin{proof} Let $(\hat M,\hat g)$ be the universal cover of $(M,g)$. The de Rham decomposition theorem shows that the universal cover $(\hat M,\hat g)$ of $(M,g)$ is isometric to a product of complete Riemannian manifolds $(\hat M_1,\hat g_1)\times(\hat M_2,\hat g_2)$ such that the lift of $T_i$ to $\hat M$ is equal to the pull-back of $\T\hat M_i$ to $\hat M$. By assumption, the pull-back $\hat \f$ of $\f$ to $\hat M$ is bounded, and $e^{2\hat \f}\hat g$ is Einstein. Using \cite[Thm. 2]{cleyton} (cf. also \cite[Corollary 3.6]{kr}) we deduce that either $\hat \f$ only depends on one factor, or $(\hat M_i,\hat g_i)$ are isometric to Euclidean spaces and $e^{-\hat\f(x,y)}=\|x\|^2+\|y\|^2+c$ for some positive constant $c$. This last case, however, is impossible since $\hat\f$ is bounded.

\end{proof}

We may thus assume from now on that $X(\f)=0$ for every vector $X$ tangent to $T_2$. This is equivalent to 
\begin{equation}\label{zero}\d \f\wedge\omega_1=0,
\end{equation}
where $\omega_1$ denotes the volume form of the distribution $T_1$. Since $\omega_1$ is $\nabla^{g}$-parallel, Lemma \ref{conf-inv} shows that $\tilde \omega_1:=e^{(p+1)\f}\omega_1$ is a conformal Killing $p$-form on $(M,\tilde g)$, where $p$ denotes the rank of $T_1$. Moreover, by \eqref{zero}
$$\d \tilde \omega_1=(p+1)e^{(p+1)\f}\d \f\wedge\omega_1=0.$$
Consequently, if $\ast_{\tilde g}$ denotes the Hodge operator on $(M,\tilde g)$, we deduce that $\tilde\omega_2:=\ast_{\tilde g}\tilde\omega_1$ is a Killing form on $(M,\tilde g)$. Moreover, by Lemma \ref{d1}, we have $\deg(\tilde\omega_2)\ge 2$.

We claim that $\tilde\omega_2$ is $\nabla^{\tilde g}$-parallel. Since $(M,\tilde g)$ is of type E, we distinguish the following cases:
\begin{itemize}
\item If $\Hol_0(M,\tilde g)$ is conjugated to $\SU(n/2)$ for $n$ even, or to $\Sp(n/4)$ for $n$ multiple of $4$, then $(M,\tilde g)$ is irreducible Ricci-flat, so the Cheeger-Gromoll theorem \cite[Cor. 6.67]{besse} shows that a finite covering of $(M,\tilde g)$ is K\"ahler, and the claim follows from Proposition \ref{prop} (i). 
\item If $\Hol_0(M,\tilde g)$ is conjugated to $\G_2$ for $n=7$ or to $\Spin(7)$ for $n=8$, then $(M,\tilde g)$ is irreducible Ricci-flat, so the Cheeger-Gromoll theorem again shows that a finite covering of $(M,\tilde g)$ has full holonomy $\G_2$ or $\Spin(7)$ and the claim follows from Proposition \ref{prop} (iii).
\item If $\Hol_0(M,\tilde g)$ is conjugated to $\Sp(n/4)\Sp(1)$ for $n\ge 8$ and multiple of $4$, then Lemma \ref{hol} (ii) shows that either $(M,\tilde g)$, or a double covering of it, is quaternion-K\"ahler, so the claim follows from from Proposition \ref{prop} (ii). 
\item If $(M,\tilde g)$ is locally symmetric, irreducible, with non-constant sectional curvature, then the claim follows from from Proposition \ref{prop} (iv).
\end{itemize}
We thus have shown that $\tilde\omega_2$ is $\nabla^{\tilde g}$-parallel. By Hodge duality, $\tilde\omega_1$ is also $\nabla^{\tilde g}$-parallel, in particular $|\tilde\omega_1|_{\tilde g}=:c$ is constant. On the other hand, $\omega_1$ has norm 1 with respect to $g$, whence
$$c=|\tilde\omega_1|_{\tilde g}=e^{-p\f}|\tilde\omega_1|_{g}=e^{-p\f}|e^{(p+1)\f}\omega_1|_{g}=e^{\f}|\omega_1|_{g}=e^\f,$$ thus showing that $\f$ is constant, contradicting the fact that $g$ and $\tilde g$ are non-homothetic. This concludes the proof of the theorem.
\end{proof}

\section{Triple warped products}

In this last section we treat the case, left open in Theorem \ref{main}, of conformally related non-homothetic metrics with reducible holonomy. With start with some necessary definitions.

\begin{defi}\label{def1} Let $(M_i,g_i)$, $i\in\{1,2,3\}$ be three connected Riemannian manifolds of positive dimension and $\f$ a (non-constant) function on $M_2$. The {\em triple warped product} associated to this data is the Riemannian manifold $M:=M_1\times M_2\times M_3$ endowed with the metric $g:=g_1+g_2+e^{-2\f}g_3$. The function $\f$ is called {\em warping function}.
\end{defi}
A triple warped product manifold is thus a Riemannian product, with one factor being itself a warped product. The nice feature of triple warped product metrics is, as noticed in the introduction, that in their conformal class there is a second (non-homothetic) triple warped product metric: 
$$\tilde g:=e^{2\f}g=e^{2\f}(g_1+g_2)+g_3=g_3+e^{2\f}g_2+e^{2\f}g_1,$$
which is the triple warped product metric associated to the Riemannian manifolds $(M_3,g_3)$, $(M_2,e^{2\f}g_2)$, $(M_1,g_1)$, and to the warping function $-\f$.

By an abuse of language, we will use the same terminology for the following slightly more general notion:

\begin{defi} A Riemannian metric $g$ on a manifold $M$ is called a triple warped product metric if there exist:
\begin{itemize}
\item a Riemannian metric $g'$ on $M$;
\item a decomposition of $\T M=T_1\oplus T_2\oplus T_3$ as a direct sum of three distributions which are mutually orthogonal with respect to $g'$ and $\nabla^{g'}$-parallel;
\item a non-constant function $\f$ on $M$ whose differential $\d\f$ vanishes on $T_1\oplus T_3$ such that $g=g_1+g_2+e^{-2\f}g_3$, where $g_i$ denotes the restriction of $g'$ to $T_i$.
\end{itemize}
\end{defi}
By the de Rham decomposition theorem, a Riemannian metric $g$ on a manifold $M$ is a triple warped product metric if and only if the universal cover $(\hat M,\hat g)$ is a triple warped product with warping function invariant by $\pi_1(M)$. 
By the above remarks, each conformal class of triple warped product metric contains two reducible Riemannian metrics. The aim of this section is to show that up to finite coverings, the converse is true in the compact case:

\begin{thm} \label{reducible} Let $M$ be a compact manifold carrying two conformally related non-homo\-thetic Riemannian metrics $g$ and $\tilde g:=e^{2\f}g$ which have both reducible holonomy. Then up to a finite covering, $g$ is a triple warped product on $M$ with warping function $\f$.
\end{thm}

\begin{proof} Lemma \ref{hol} (iii) shows that by replacing $M$ with a finite covering if necessary, there exists a $g$-orthogonal $\nabla^g$-parallel decomposition $\T M=T_1\oplus T_{23}$ and a $\tilde g$-orthogonal $\nabla^{\tilde g}$-parallel decomposition $\T M=T_{12}\oplus T_3$ (the notations will start making sense later on). Moreover, up to a change in notations, we may assume that $\rk(T_1)\le\rk(T_{12})$ and $\rk(T_1)\le\rk(T_3)$. 

For the moment being, we have no information about the relative position of these distributions in $\T M$. However, using again the theory of conformal Killing forms, we will show that in fact $T_1$ is contained either in $T_{12}$ or in $T_3$.

To see this, let us denote by $p$ the rank of $T_1$, and consider the volume form $\psi$ of $T_1$ defined by $g$. Since $\psi$ is parallel on $(M,g)$, Lemma \ref{conf-inv} shows that $\tilde\psi:=e^{(p+1)\f}\psi$ is a conformal Killing $p$-form on the reducible manifold $(M,\tilde g)$. By Proposition \ref{ckp}, $\tilde\psi$ can be written as a sum  $\tilde\psi=\tilde\psi_1+\tilde\psi_2+\tilde\psi_3+\tilde\psi_4+\tilde\psi_5$, where:
\begin{itemize}
\item $\tilde\psi_1$ is a parallel $p$-form on $(M,\tilde g)$.
\item $\tilde\psi_2$ is a Killing $p$-form on $(M,\tilde g)$ which is basic with respect to $T_{12}$ and $\tilde\psi_3$ is a Killing $p$-form on $(M,\tilde g)$ which is basic with respect to $T_{3}$.
\item $\tilde\psi_4$ is the Hodge dual on $(M,\tilde g)$ of a Killing $(n-p)$-form on $(M,\tilde g)$ which is basic with respect to $T_{12}$ and $\tilde\psi_5$ is the Hodge dual on $(M,\tilde g)$ of a Killing $(n-p)$-form on $(M,\tilde g)$ which is basic with respect to $T_{3}$.
\end{itemize}

From our dimensional assumption, $n-p\ge\rk(T_{12})$ and $n-p\ge\rk(T_3)$, and since every basic Killing form of maximal degree is automatically closed, thus parallel, we deduce that $\tilde\psi_4$ and $\tilde\psi_5$ are parallel (or vanish identically if the inequalities are strict), so finally $\tilde\psi$ is the sum of parallel and Killing forms on $(M,\tilde g)$, in particular it is co-closed (i.e. in the kernel of $\delta^{\tilde g}$). 

Using the general formula \eqref{deltapsi} relating the co-differentials of $g$ and $\tilde g$, and the fact that $\psi$ is parallel on $(M,g)$, we deduce that $(\grad_g\f)\lrcorner \psi=0$, thus showing that $\d\f(X)=0$ for every $X\in T_1$.

Consider now the volume form $\tilde\omega$ of the distribution $T_3$ with respect to $\tilde g$. We denote by $q$ the rank of $T_3$. Since $\tilde\omega$ is parallel on $(M,\tilde g)$, Lemma \ref{conf-inv} shows that $\omega:=e^{-(q+1)\f}\tilde\omega$ is conformal Killing on $(M,g)$. By Proposition \ref{ckp} again, $\omega$ can be written as a sum  $\omega=\omega_1+\omega_2+\omega_3+\omega_4+\omega_5$, where:
\begin{itemize}
\item $\omega_1$ is a parallel $q$-form on $(M,g)$.
\item $\omega_2$ is a Killing $q$-form on $(M,g)$ which is basic with respect to $T_1$, and $\omega_3$ is a Killing $q$-form on $(M,g)$ which is basic with respect to $T_{23}$
\item $\omega_4$ is the Hodge dual on $(M, g)$ of a Killing $(n-q)$-form on $(M,g)$ which is basic with respect to $T_1$, and $\omega_5$ is the Hodge dual on $(M, g)$ of a Killing $(n-q)$-form on $(M,g)$ which is basic with respect to $T_{23}$.
\end{itemize}

Like before, the dimensional assumption $p\le q$ and $p\le n-q$ show that $\omega_2$ and $\omega_4$ are parallel (or vanish identically if the inequalities are strict). We deduce that $\delta^g\omega=\delta^g\omega_5$. Let $\omega_6$ be the Killing form on $(M,g)$ which is basic with respect to $T_{23}$ such that $\omega_5=*_g\omega_6$. From \eqref{identity} we get $\delta^g\omega_5= \omega_7\wedge\psi$, where $\omega_7$ is basic with respect to $T_{23}$ (recall that $\psi$ is the volume form of $T_1$ with respect to $g$). On the other hand, from \eqref{deltapsi} and the fact that $\tilde\omega$ is parallel (thus co-closed) on $(M,\tilde g)$, we get $\delta^g\omega=(n-q+1)(\grad_g\f)\lrcorner\omega$, so finally
\begin{equation}\label{omega}\omega_7\wedge\psi=(n-q+1)(\grad_g\f)\lrcorner\omega.\end{equation}

Repeating this argument, this time starting with the volume form $\tilde\sigma$ of $T_{12}$ with respect to $\tilde g$, and denoting by $\sigma:=e^{-(n-q+1)\f}\tilde\sigma$, we obtain the existence of a form $\sigma_7$ which is basic with respect to $T_{23}$ such that 
\begin{equation}\label{sigma}\sigma_7\wedge\psi=(q+1)(\grad_g\f)\lrcorner\sigma.\end{equation}

We now consider the following closed subsets of $M$:
$$M_1:=\{x\in M\ |\ (\grad_g\f)_x\in T_3\},\qquad M_2:=\{x\in M\ |\ (\grad_g\f)_x\in T_{12}\},$$
and 
$$C_1:=\{x\in M\ |\ (T_1)_x\subset (T_3)_x\},\qquad C_2:=\{x\in M\ |\ (T_1)_x\subset (T_{12})_x\}.$$

Since $(\grad_g\f)\lrcorner\omega\in \Lambda^{q-1}T_3^*$, $(\grad_g\f)\lrcorner\sigma\in \Lambda^{n-q-1}T_{12}^*$, and their exterior product vanishes by \eqref{omega} and \eqref{sigma}, we deduce that at each point of $x\in M$, either $(\grad_g\f)_x\lrcorner\omega_x=0$ (i.e. $(\grad_g\f)_x\in T_{12}$), or $(\grad_g\f)_x\lrcorner\sigma_x=0$ (i.e. $(\grad_g\f)_x\in T_{3}$). This just means that $M_1\cup M_2=M$. 

For every $x\in M_1\setminus M_2$ we have $(\grad_g\f)_x\in T_3\setminus\{0\}.$ Using \eqref{omega} we can write at $x$:
$$\omega=\frac{\d\f}{\d\f(\grad_g\f)}\wedge (\grad_g\f)\lrcorner\omega=\frac{\d\f}{(n-q+1)\d\f(\grad_g\f)}\wedge\omega_7\wedge\psi.$$
As $\omega$ is a volume form of $T_3$ and $\psi$ is a volume form of $T_1$, this shows that $(T_1)_x\subset (T_3)_x$, whence $M_1\setminus M_2\subset C_1.$ Since $C_1$ is closed, we also have 
\begin{equation}\label{m1}\overline{M_1\setminus M_2}\subset C_1.\end{equation}
Similarly, we obtain using \eqref{sigma} that 
\begin{equation}\label{m2}\overline{M_2\setminus M_1}\subset C_2.\end{equation}

\begin{lemma}
The interior of $M_1\cap M_2$ is contained in $C_1\cup C_2$. 
\end{lemma}
\begin{proof} For every $x\in \mathrm{Int}(M_1\cap M_2)$ there exists a smooth path $c:[0,1]\to M$ such that $c(0)=x$, $c([0,1))\subset  \mathrm{Int}(M_1\cap M_2)$ and $$y:=c(1)\in (M_1\cap M_2)\setminus \mathrm{Int}(M_1\cap M_2)\subset( \overline{M_1\setminus M_2})\cup (\overline{M_2\setminus M_1})\subset C_1\cup C_2.$$

Now, since $\d\f=0$ on $M_1\cap M_2$, the Levi-Civita connections of $g$ and $\tilde g$ are the same on $M_1\cap M_2$, so the parallel transport along $c$ with respect to $\nabla^g$ coincides with the parallel transport with respect to $\nabla^{\tilde g}$. If $y\in C_1$, we have by definition $(T_1)_y\subset (T_3)_y$ and since $T_1$ is $\nabla^g$-parallel and $T_3$ is $\nabla^{\tilde g}$-parallel, we obtain that $T_1\subset T_3$ at each point $c(t)$, in particular for $t=0$ which shows that $x\in C_1$. Similarly, if $y\in C_2$ we obtain $x\in C_2$, thus proving the lemma.
\end{proof}
As a consequence of the fact that $C_1$ and $C_2$ are closed, we thus obtain:
\begin{equation}\label{m3}\overline{\mathrm{Int}(M_1\cap M_2)}\subset C_1\cup C_2.\end{equation}
Now, for every closed sets $M_1$ and $M_2$ with $M_1\cup M_2=M$ one has 
$$(\overline{M_1\setminus M_2})\cup(\overline{M_2\setminus M_1})\cup(\overline{\mathrm{Int}(M_1\cap M_2)})=M.$$
From \eqref{m1}, \eqref{m2} and \eqref{m3} we get $C_1\cup C_2=M$. On the other hand, $C_1$ is obviously disjoint from $C_2$, so by the connectedness of $M$ we have $C_1=M$ or $C_2=M$. Up to a switch of notations between $T_3$ and $T_{12}$, we can therefore assume that $C_1=\emptyset$ and $C_2=M$, i.e. $T_1$ and $T_3$ are orthogonal at each point of $M$. Moreover, from \eqref{m1} we get $M_1\subset M_2$, i.e. $M_2=M$, thus showing that $\d\f(X)=0$ for every $X\in T_3$.

We now define $T_2$ as the orthogonal complement with respect to $g$ (or $\tilde g$) of $T_1\oplus T_3$. Since $T_2=T_1^\perp\cap T_3^\perp$ is the intersection of two integrable distributions, it is integrable. We have thus shown that the tangent bundle of $M$ decomposes in a direct sum $\T M=T_1\oplus T_2\oplus T_3$ of integrable, mutually orthogonal distributions, and that the differential of the conformal factor $\f$ vanishes on $T_1$ and $T_3$. 

Since $T_1$ and $T_{23}$ are $\nabla^g$-parallel orthogonal distributions, the restrictions $g_1$ and $g_{23}$ of $g$ to $T_1$ and $T_{23}$ are basic with respect to $T_1$ and $T_{23}$ respectively, and $g=g_1+g_{23}$. Similarly, one can write $e^{2\f}g=\tilde g=g_{12}+g_3$, where $g_{12}$ and $g_3$ are basic symmetric bilinear forms on $T_{12}$ and $T_3$ respectively. We thus get $g_1+g_{23}=e^{-2\f}(g_{12}+g_3)$, which also reads:
$$g_{23}-e^{-2\f}g_3=e^{-2\f}g_{12}-g_1.$$
Since the left hand side vanishes on $T_1$ and the right hand side vanishes on $T_3$, we see that the symmetric tensor $g_2:=g_{23}-e^{-2\f}g_3=e^{-2\f}g_{12}-g_1$ vanishes on $T_1\oplus T_3$. We thus have $g_{23}=e^{-2\f}g_3+g_2$, whence 
\begin{equation}\label{g}g=g_1+g_2+e^{-2\f}g_3,\end{equation} 
and in particular $g_i$ is a positive definite symmetric bilinear form on $T_i$ for $i=1,2,3$. 

Moreover, the expression $g_2=g_{23}-e^{-2\f}g_3$ shows that $g_2$ is constant in the directions of $T_1$ (in the sense that the Lie derivative of $g_2$ in the direction of vector fields tangent to $T_1$ vanishes). Similarly, the formula $g_2=e^{-2\f}g_{12}-g_1$ shows that $g_2$ is constant in the directions of $T_3$. Consequently, $g_2$ is basic with respect to $T_2$ and the theorem is proved.
\end{proof}

\end{document}